\documentclass[12pt]{amsart}

\usepackage{times,amsfonts,amsmath,amstext,amsbsy,amssymb,
  amsopn,amsthm,upref,eucal, amscd, graphicx}
\usepackage[pdftex]{color}
\usepackage[T1]{fontenc}

\newtheorem{theorem}{Theorem}[section]
\newtheorem{lemma}[theorem]{Lemma}

\newtheorem{proposition}[theorem]{Proposition}

\numberwithin{equation}{section}

\theoremstyle{definition}

\newtheorem{remark}[theorem]{Remark}

\begin{document}
\title[On stable sets of non-uniformly hyperbolic horseshoes]{An estimate on the Hausdorff dimension of stable sets of non-uniformly hyperbolic horseshoes}
\author{Carlos Matheus}
\address{Carlos Matheus: Universit\'e Paris 13, Sorbonne Paris Cit\'e, LAGA, CNRS (UMR 7539), F-93439, Villetaneuse, France}
\email{matheus@impa.br.}
\author{Jacob Palis}
\address{Jacob Palis: IMPA, Estrada D. Castorina, 110, CEP 22460-320, Rio de Janeiro, RJ, Brazil}
\email{jpalis@impa.br.}
\date{August 22, 2017}
\begin{abstract}
We show that the Hausdorff dimension of stable sets of non-uniformly hyperbolic horseshoes is strictly smaller than two.  
\end{abstract}
\maketitle

\tableofcontents

\section{Introduction}\label{intro}

We study in this article the geometry of stable and unstable sets of the 
\emph{non-uniformly hyperbolic horseshoes} 
introduced by Palis and Yoccoz in \cite{PY09} that appear very frequently in heteroclinic bifurcations associated to ``slightly thick'' horseshoes. 

More concretely, our main goal is to estimate the Hausdorff dimensions of the stable and unstable sets of non-uniformly hyperbolic horseshoes. 

Before stating our main result 
namely, Theorem \ref{t.MPY-A} below, 
we review some statements from 
\cite{PY09}.

\subsection{Heteroclinic bifurcations of slightly thick horseshoes}\label{ss.PYsetting-intro} Let $g_0:M\to M$ be a smooth ($C^{\infty}$) diffeomorphism of a compact surface $M$ displaying the following dynamical features. 

We suppose that $g_0$ possesses a horseshoe $K$ containing two periodic points $p_s$ and $p_u$ involved in a \emph{heteroclinic tangency}, that is, the points $p_s, p_u\in K$ belong to distinct periodic orbits and the invariant manifolds $W^s(p_s)$ and $W^u(p_u)$ meet \emph{tangentially} at some point $q\in M-K$. 

Also, we assume that this heteroclinic tangency is \emph{quadratic}, i.e., the curvatures of $W^s(p_s)$ and $W^u(p_u)$ at $q$ are distinct. 

Moreover, we suppose that the heteroclinic tangency is the sole responsible for the local dynamics of $g_0$ near $K$ and $q$, that is, 
there are neighborhoods $U$ of $K$ and $V$ of the orbit $\mathcal{O}(q)$ such that $K\cup\mathcal{O}(q)$ is the maximal invariant set of $U\cup V$.  

\begin{figure}[htb!]
\includegraphics[scale=0.6]{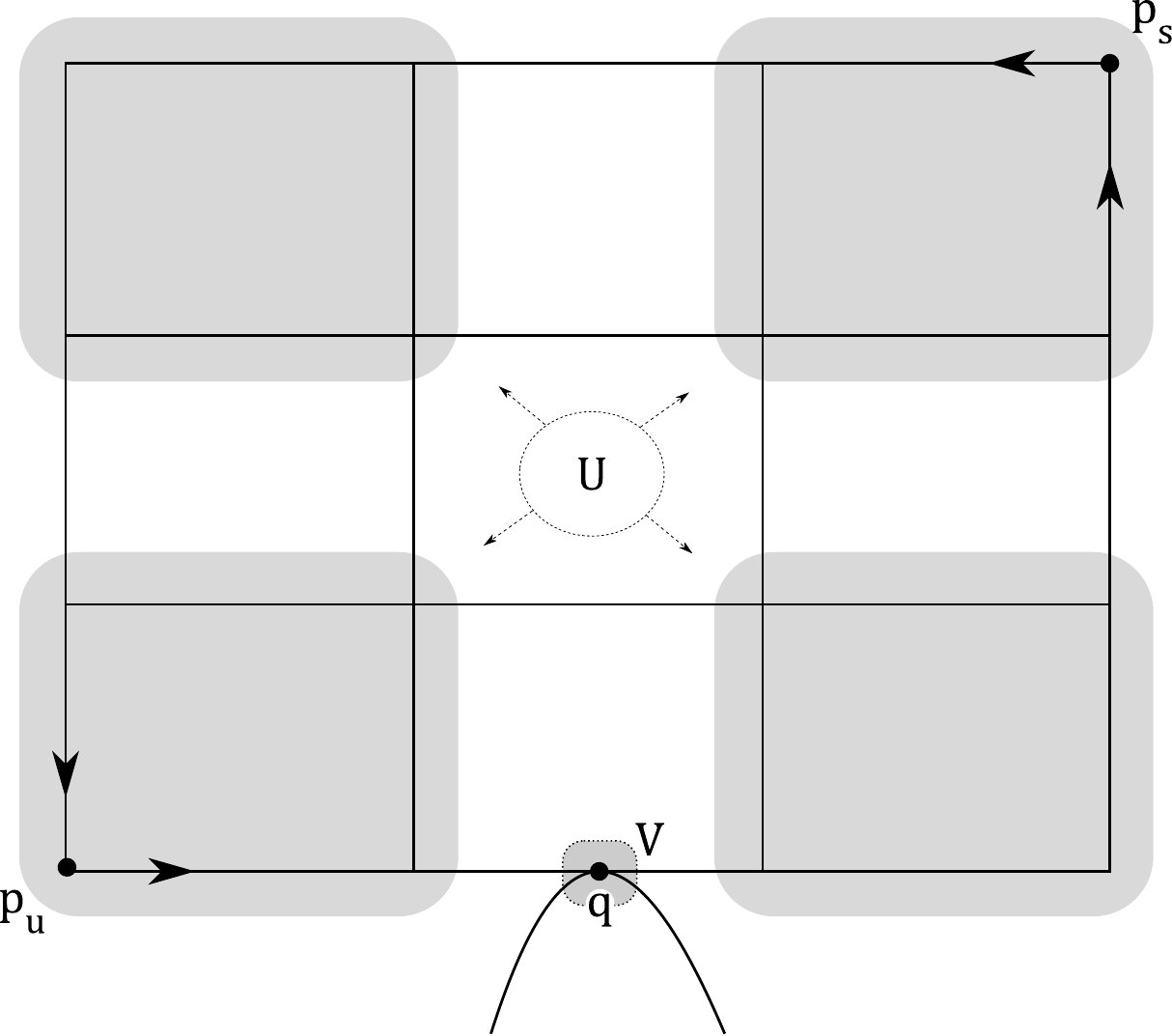}
\caption{Local dynamics near a heteroclinic tangency.}\label{f.heteroclinic-tangency}
\end{figure}

Consider $(g_t)_{|t|<t_0}$ a $1$-parameter family of smooth diffeomorphisms of $M$ \emph{generically}\footnote{
This means that the quadratic tangency between $W^s(p_s)$ and $W^u(p_u)$ move with positive speed when the parameter $t$ varies. See Section \ref{s.preliminaries} for the precise definition.} unfolding the heteroclinic tangency of $g_0$ in such a way that the continuations of adequate compact pieces of $W^s(p_s)$ and $W^u(p_u)$ have no intersection near $q$ for $-t_0<t<0$ and two transverse intersections near $q$ for $0<t<t_0$. 

The long-term goal is to understand the local dynamics of $g_t$, $t\in (-t_0, t_0)$, near $K$ and $q$. More precisely, we are interested in the features of the maximal invariant set 
\begin{equation}\label{e.Lambda-gt}
\Lambda_{g_t}:= \bigcap\limits_{n\in\mathbb{Z}} g_t^{-n}(U\cup V)
\end{equation}
where $U\cup V$ is the neighborhood of $K\cup\mathcal{O}(q)$ described above.  

Note that the maximal invariant set 
\begin{equation}\label{e.K-gt}
K_{g_t}:=\bigcap\limits_{n\in\mathbb{Z}} g_t^{-n}(U)
\end{equation}
is a horseshoe corresponding to the natural (hyperbolic) continuation of $K$.

It is not hard to see that $\Lambda_{g_t} = K_{g_t}$ when $-t_0<t<0$. Since $\Lambda_{g_0}=K\cup\mathcal{O}(q)$, we have that the set $\Lambda_{g_t}$ is \emph{not} dynamically interesting when $-t_0<t\leq 0$. 

Given this scenario, it is natural to try to understand $\Lambda_{g_t}$ for $0<t<t_0$. In this direction, 
it is introduced in \cite{PY09} a notion of \emph{strongly regular parameters} $t$ with the property that $\Lambda_{g_t}$ is a \emph{non-uniformly hyperbolic horseshoe} 
whenever $t$ is a strongly regular parameter. Furthermore, 
it is showed in \cite{PY09} that \emph{most} parameters are strongly regular in the sense that 
$$\lim\limits_{s\to0^+}\frac{1}{s}\textrm{Leb}_1(\{0<t\leq s: t \textrm{ is a strongly regular parameter}\})=1$$
for heteroclinic tangencies associated to \emph{slightly thick} horseshoes, i.e., when the initial horseshoe $K_{g_0}$ satisfies 
\begin{equation}\label{e.PYset}
(d_s^0+d_u^0)^2+(\max\{d_s^0, d_u^0\})^2< (d_s^0+d_u^0) + \max\{d_s^0, d_u^0\}
\end{equation} 
where $d_s^0$ and $d_u^0$ (resp.) are the transverse Hausdorff dimensions of the invariant sets $W^s(K_{g_0})$ and $W^u(K_{g_0})$ (resp.). Here $\textrm{Leb}_1$ is the $1$-dimensional Lebesgue measure.

In particular, these results 
in \cite{PY09} imply that, by generically unfolding a heteroclinic tangency associated to a slightly thick horseshoe, the maximal invariant set $\Lambda_{g_t}$ is a non-uniformly hyperbolic horseshoe for most parameters $t>0$ near $0$ in the sense that the density of parameters $t\in (0,s]$ such that $\Lambda_{g_t}$ is a non-uniformly hyperbolic horseshoe tends to $1$ as $s\to 0^+$. 


Among the several geometrical features of non-uniformly hyperbolic horseshoes shown in \cite{PY09}, we recall that:
\begin{theorem}[cf. Theorem 6 in \cite{PY09}]\label{t.PY09-Thm6} Assuming \eqref{e.PYset}, if $t$ is a strongly regular parameter, then 
$$\textrm{Leb}_2(W^s(\Lambda_{g_t}))=\textrm{Leb}_2(W^u(\Lambda_{g_t}))=0,$$
where $\textrm{Leb}_2$ is the $2$-dimensional Lebesgue measure. In particular, $\Lambda_{g_t}$ does not contain attractors nor repellors.
\end{theorem}

In other words, there is no abrupt \emph{explosion} of the local dynamics of $g_t$ on $\Lambda_{g_t}$ for most parameters $t$, 
that is, for any strongly regular parameter $t$, the stable, resp. unstable, set $W^s(\Lambda_{g_t})$, resp. $W^u(\Lambda_{g_t})$, of the non-uniformly hyperbolic horseshoe $\Lambda_{g_t}$ are ``small'' (zero $2$-dimensional Lebesgue measure). 

Actually, it was conjectured in \cite[p. 14]{PY09} that the stable, resp. unstable, sets of non-uniformly hyperbolic horseshoes are \emph{really} small: their Hausdorff dimensions should be \emph{strictly} smaller than $2$, and, in fact, 
close 
to the ``expected'' dimension $1+d_s$, resp. $1+d_u$, where $d_s$, resp. $d_u$, is a 
quantity introduced in \cite{PY09} (close to $d_s^0$, resp. $d_u^0$) measuring the transverse dimension of the stable, resp. unstable, set of the ``main non-uniformly hyperbolic part'' of 
$\Lambda_{g_t}$.

\subsection{Statement of the main result} Our theorem confirms the first part of the conjecture stated above.

\begin{theorem}\label{t.MPY-A} Consider the setting of the previous subsection \ref{ss.PYsetting-intro} of a $1$-parameter family $(g_t)_{|t|<t_0}$ generically unfolding a heteroclinic tangency associated to two distinct periodic orbits belonging to a initial slightly thick horseshoe $K_{g_0}$ in the sense of \eqref{e.PYset} above. 

Then, for any strongly regular parameter $t$, one has 
$$\textrm{HD}(W^s(\Lambda_{g_t}))<2 \quad \textrm{ and } \quad \textrm{HD}(W^u(\Lambda_{g_t}))<2$$
where $\textrm{HD}$ stands for the Hausdorff dimension.
\end{theorem}

\begin{remark} It is worth to point out that we prove Theorem \ref{t.MPY-A} for the \emph{same} strongly regular parameters of Palis-Yoccoz \cite{PY09}, but our arguments involve only \emph{soft analysis} instead of the parameter exclusion methods in \cite{PY09}. 
\end{remark}

\begin{remark} Our proof of Theorem \ref{t.MPY-A} does not allow us to settle the second part of the previous conjecture (on the exact value of the Hausdorff dimension of stable sets of non-uniformly hyperbolic horseshoes): see Remark 
\ref{r.expected-dimension-MPY-A} below for more explanations. 
\end{remark}

\subsection{Outline of the proof of the main result} The stable $W^s(\Lambda)$ of a non-uniformly hyperbolic horseshoe $\Lambda$ is naturally decomposed into a ``well-behaved'' part and an ``exceptional'' part (cf. Subsection 11.6 of \cite{PY09} and/or Section \ref{s.preliminaries} below). 

Roughly speaking, the well-behaved part of $W^s(\Lambda)$ consists of points captured by ``stable curves'' obtained as the intersections of decreasing sequences of certain domains (``strips'') where adequate iterates of the dynamics behave like ``affine hyperbolic maps'' (in appropriate coordinates). 

From these features of the dynamics on the well-behaved part of $W^s(\Lambda)$, it is possible to show that its decomposition into stable curves is a lamination with $C^{1+Lip}$-leaves and Lipschitz holonomy (cf. Subsection 10.5 of \cite{PY09}) whose transverse Hausdorff dimension $0<d_s<1$ is close to the stable dimension $d_s^0$ of the initial horseshoe $K_{g_0}$ (cf. Theorem 4 in Subsection 10.10 of \cite{PY09}). 

In particular, the well-behaved part of $W^s(\Lambda)$ has Hausdorff dimension $1+d_s$, and, hence, our task consists into studying the geometry of the exceptional part $\mathcal{E}^+$ of $W^s(\Lambda)$. 

In other words, the proof of Theorem \ref{t.MPY-A} is reduced to show that the Hausdorff dimension of the exceptional part $\mathcal{E}^+$ of $W^s(\Lambda)$ is $HD(\mathcal{E}^+)<2$.  

By definition, the exceptional part $\mathcal{E}^+$ of $W^s(\Lambda)$ consists of points whose forward orbits get ``very close'' to the ``critical locus'' (of tangencies) infinitely many times. In fact, between ``affine-like hyperbolic'' iterations, the forward orbit of a point in $\mathcal{E}^+$ visits a sequence of domains (``parabolic cores $c(P_k)$ of strips $P_k$'') close to the critical locus whose ``widths'' decay with a \emph{double exponential} rate (cf. Lemma 24 of \cite{PY09}). 

The scenario described in the previous paragraph imposes \emph{strong geometrical constraints} on $\mathcal{E}^+$. For the sake of comparison, it is worth to point out that the forward orbit of a point in the stable set of a \emph{uniformly} hyperbolic horseshoe visits a sequence of ``strips'' (cylinders of a Markov partition) whose ``widths'' decay with an \emph{exponential} rate. In particular, this suggests that $\mathcal{E}^+$ is very small when compared with the well-behaved part of $W^s(\Lambda)$. 

We show that $HD(\mathcal{E}^+)<2$ by combining the geometrical constraints on its forward iterates described above with the following simple argument. 

We know that, between affine-like iterations, the forward images of a point in $\mathcal{E}^+$ under the dynamics fall in a sequence of strips $P_k$, $k\in\mathbb{N}$, whose widths decay with a double exponential rate. By fixing $k\in\mathbb{N}$ large and decomposing the strip $P_k$ into squares, we obtain a covering of very small diameter of the image of $\mathcal{E}^+$ under some positive iterate of the dynamics. 

Now, we want to use negative iterates of the dynamics to bring this covering of $P_k$ back to $\mathcal{E}^+$. On one hand, we observe that each square becomes a strip under affine-like iterates of the dynamics. On the other hand, these strips might get folded during non-affine-like iterations (when the strips visit the parabolic cores of $P_j$). Since these folding effects can accumulate very quickly, it is not easy to keep control of their fine geometry in our way back from $P_k$ to $\mathcal{E}^+$. 

Fortunately, if one wants just to prove that $HD(\mathcal{E}^+)<2$, then we can simply ``forget'' about the fine details of the geometries of these folded strips inside the $P_j$, $0\leq j\leq k$, by thinking of them as ``fat strips''. In other words, when the strips acquire ``parabolic shapes'' due to the folds occuring inside $P_j$'s, we treat these ``parabolic shapes'' as \emph{new} strips, we decompose them into \emph{new} squares and we bring back each of these squares \emph{individually} under the dynamics. Of course, the number of squares increases \emph{significantly} each time we perform this procedure, but we will see that the resulting cover of $\mathcal{E}^+$ has a \emph{mild} cardinality (in comparison with its diamater) thanks to the double exponential decay of the widths of $P_j$'s and the affine-like features of the dynamics between consecutive passages through the parabolic cores of $P_j$'s. In particular, by letting $k\in\mathbb{N}$ vary, this argument will provide a sequence of covers of $\mathcal{E}^+$ whose diameters approach zero allowing to conclude that $HD(\mathcal{E}^+)<2$, and, \emph{a fortiori}, the proof of Theorem \ref{t.MPY-A}.


\subsection{Organization of the article} The remainder of this paper is divided in two sections. In Section \ref{s.preliminaries}, we will recall for later use some material from the article \cite{PY09} including the basic features of non-uniformly hyperbolic horseshoes. After that, we prove Theorem \ref{t.MPY-A} in Section \ref{s.MPY-A}.   

\subsection*{Acknowledgments} This text is much influenced by our forthcoming joint work \cite{MPY} with Jean-Christophe Yoccoz: we were very fortunate to have known and worked with him. We are also grateful to the following institutions for their hospitality during the preparation of this article: Coll\`ege de France, Instituto de Matem\'atica Pura e Aplicada (IMPA), and Kungliga Tekniska H\"ogskolan (KTH). The authors were partially supported by the Balzan Research Project of J. Palis and the French ANR grand ``DynPDE'' (ANR-10-BLAN 0102). Last, but not least, we are thankful to the referee for carefully reading this article.  

\section{Preliminaries}\label{s.preliminaries}

In this section, we will briefly review some of the main features of the non-uniformly horseshoes introduced in \cite{PY09}. 

\subsection{Heteroclinic bifurcations}

Let $f:M\to M$ be a smooth ($C^{\infty}$) diffeomorphism of a compact surface $M$. Suppose that $f$ possesses a horseshoe $K$ containing two periodic points $p_s$ and $p_u$ involved in a \emph{quadratic heteroclinic tangency}, that is, the periodic points $p_s, p_u\in K$ belong to distinct periodic orbits, the invariant manifolds $W^s(p_s)$ and $W^u(p_u)$ meet \emph{tangentially} at some point $q\in M-K$, and the curvatures of $W^s(p_s)$ and $W^u(p_u)$ at $q$ are distinct. Moreover, we assume that there are neighborhoods $U$ of $K$ and $V$ of the orbit $\mathcal{O}(q)$ such that $K\cup\mathcal{O}(q)$ is the maximal invariant set of $U\cup V$. See Figure \ref{f.heteroclinic-tangency} above. 

In the sequel, $\mathcal{U}$ denots a sufficiently small neighborhood of the diffeomorphism $f$ (in $\textrm{Diff}^{\infty}(M)$) such that all relevant dynamical objects admit natural continuations. In particular, we will assume that the dynamical objects introduced above (namely, $p_s, p_u, K$) admit natural (hyperbolic) continuations for $g\in\mathcal{U}$.

Given $g\in\mathcal{U}$, we have exactly three (mutually exclusive) possibilities for the intersection of some appropriate compact pieces of the continuations of $W^s(p_s)$ and $W^u(p_u)$ near $q$: they meet at no point, they meet tangentially at one point or they meet transversely at two points. By definition: 
\begin{itemize}
\item $g\in\mathcal{U}_-$ in the first case (of no intersection near $q$), 
\item $g\in\mathcal{U}_0$ in the second case (of one tangential intersection near $q$), 
\item $g\in\mathcal{U}_+$ in the third case (of two transverse intersections near $q$).
\end{itemize}
In particular, $\mathcal{U}_0$ is a codimension $1$ submanifold, and we have that $\mathcal{U}=\mathcal{U}_-\cup\mathcal{U}_0\cup\mathcal{U}_+$. 

We wish to understand the local dynamics of $g$ near $K$ and $q$ for $g\in\mathcal{U}$. More precisely,  let $(g_t)_{|t|<t_0}$ be a $1$-parameter family \emph{generically} unfolding the heteroclinic tangency of $g_0$ (associated to the continuations of the periodic points $p_s$ and $p_u$), that is, $(g_t)_{|t|<t_0}$ satisfies: 
\begin{itemize}
\item $g_0\in\mathcal{U}_0$,  
\item $g_t\in\mathcal{U}_+$ for $t>0$, and
\item $(g_t)_{|t|<t_0}$ is transverse to $\mathcal{U}_0$ at $g_0$.
\end{itemize}
In this setting, we want to describe the features of the maximal invariant set 
\begin{equation}\label{e.Lambda-g}
\Lambda_{g_t}:= \bigcap\limits_{n\in\mathbb{Z}} g_t^{-n}(U\cup V)
\end{equation}
where $U\cup V$ is the neighborhood of $K\cup\mathcal{O}(q)$ described above.  

\begin{remark} The set 
\begin{equation}\label{e.K-g}
K_{g_t}:=\bigcap\limits_{n\in\mathbb{Z}} g_t^{-n}(U)
\end{equation}
is a horseshoe of $g\in\mathcal{U}$.
\end{remark}

It is not hard to see that $\Lambda_{g_t}=K_{g_t}$ is a horseshoe when $-t_0<t<0$ and $\Lambda_{g_0}=K_{g_0}\cup\mathcal{O}(q_{g_0})$ for $t=0$ (where $q_{g_0}$ is the tangency point near $q$ referred to in the definition of $\mathcal{U}_0$). 

In other terms, the set $\Lambda_{g_t}$ is \emph{not} dynamically interesting when $-t_0<t\leq 0$, and, thus, one can focus exclusively on the sets $\Lambda_{g_t}$ for $0<t<t_0$. 

\subsection{Strongly regular parameters}\label{ss.strong-regular-1} Up to a reparametrization, we can assume that the parameter coordinate $t$ of the $1$-parameter family $(g_t)_{|t|<t_0}$ is normalized by the relative speed at the quadratic tangency: in other words, $t$ is the (oriented) distance between a piece of $W^u(p_u)$ near $q_{g_0}$ and the tip of the parabolic arc consisting of a piece of $W^s(p_s)$ near $q$ (see Section 4 of \cite{PY09} for more details). 

The \emph{strongly regular parameters} in \cite{PY09} are defined via an inductive scheme. Roughly speaking, we consider two very small constants $0<\varepsilon_0\ll\tau\ll 1$ and we define a sequence of scales $\varepsilon_{k+1} = \varepsilon_k^{1+\tau}$, $k\in\mathbb{N}$. The inductive scheme begins with the \emph{candidate} interval $I_0=[\varepsilon_0, 2\varepsilon_0]$. At the $k$th step of the inductive scheme, we divide the selected candidate intervals of the previous step into $\lfloor\varepsilon_k^{-\tau}\rfloor$ disjoint candidates of lengths $\varepsilon_{k+1}$. Then, each of these candidate intervals $I_k$ passes a \emph{strong regularity test}: the candidates passing the test are selected for the next ($(k+1)$th) step while the candidates failing the test are discarded. 

We will discuss some of the requirements in the strong regularity tests later, but for now let us mention that the precise definition of these tests in \cite{PY09} \emph{depends} on the condition \eqref{e.PYset} on the stable dimension $d_s^0$ and the unstable dimension $d_u^0$ of the horseshoe $K_{g_0}$, i.e., 
$$(d_s^0+d_u^0)^2+(\max\{d_s^0, d_u^0\})^2< (d_s^0+d_u^0) + \max\{d_s^0, d_u^0\}$$
For this reason, we will always assume that the initial horseshoe $K_{g_0}$ is slightly thick in the sense that the condition \eqref{e.PYset} is satisfied.

In this setting, the \emph{strongly regular parameters} $t\in I_0 = [\varepsilon_0,2\varepsilon_0]$ are defined as those parameters belonging to a decreasing sequence of candidate intervals passing the strong regularity tests, i.e., $\{t\}=\bigcap\limits_{k\in\mathbb{N}} I_k$ where $I_k$ are a selected candidate interval for the $(k+1)th$ step with $I_{k+1}\subset I_k$.

The ``non-uniform hyperbolicity'' of $\Lambda_{g_t}$ for strongly regular parameters $t\in I_0$ is ensured by the (very intricate) nature of the strong regularity tests applied to the candidate intervals $I_k$ with $\bigcap\limits_{k\in\mathbb{N}} I_k=\{t\}$. We will come back to this point later. 

Of course, the notion of strong regularity tests in \cite{PY09} is interesting for at least two reasons. Firstly, it is sufficiently rich to guarantee several nice properties of ``non-uniform hyperbolicity'' of $\Lambda_{g_t}$ for strongly regular parameters. Secondly, it is also a sufficiently mild constraint satisfied by a set of large measure of parameters. More precisely, it is shown in Corollary 15 of \cite{PY09} that the set of strongly regular parameters $t\in I_0=[\varepsilon_0, 2\varepsilon_0]$ has Lebesgue measure $\varepsilon_0(1 - 3\varepsilon_0^{\tau^2})$. 

Before describing the nature of strong regularity tests, we will need the preparatory material from the next three subsections where an adequate class $\mathcal{R}(I)$ of \emph{affine-like iterates} of $g$ will be attached to each candidate interval $I$. 

\subsection{Localization of the dynamics}\label{ss.local-dynamics}

We fix \emph{geometrical Markov partitions} of the horseshoes $K_g$ depending smoothly on $g\in\mathcal{U}$. In other terms, we choose a finite system of smooth charts $I_a^s\times I_a^u\to R_a\subset M$ indexed by a finite alphabet $a\in\mathcal{A}$ with the properties that these charts depend smoothly on $g\in\mathcal{U}$, the intervals $I_a^s$ and $I_a^u$ are compact, the rectangles $R_a$ are disjoint, the horseshoe $K_g$ is the maximal invariant set of $R:=\bigcup\limits_{a\in\mathcal{A}}R_a$, the family $(K_g\cap R_a)_{a\in\mathcal{A}}$ is a Markov partition of $K_g$ for $g\in\mathcal{U}$, and the boundaries of the rectangles $R_a$ are pieces of stable and unstable manifolds of periodic points in $K_g$. Moreover, we assume that no rectangle meets the orbits of $p_s$ and $p_u$ at the same time. See Figure \ref{f.heteroclinic-tangency} above. 

In this context, we have that the Markov partition $(K_g\cap R_a)_{a\in\mathcal{A}}$ provides a topological conjugacy between the dynamics of $g$ on $K_g$ and the subshift of finite type of $\mathcal{A}^{\mathbb{Z}}$ whose transitions are 
$$\mathcal{B}:=\{(a,a')\in\mathcal{A}^2: f(R_a)\cap R_{a'}\cap K_f\neq\emptyset\}.$$

Next, we observe that, for each $g\in\mathcal{U}_+$, we have a compact lenticular region $L_u\subset R_{a_u}$ (near the initial heteroclinic tangency point $q\in M-K$ of $f$)
whose boundary is the union of a piece of the unstable manifold of $p_u$ and a piece of the stable manifold of $p_s$. Furthermore, since no rectangle meets both orbits of $p_s$ and $p_u$, the lenticular region $L_u$ travels outside $R$ for $N_0-1$ iterates of $g\in\mathcal{U}_+$ before entering $R$ (for some integer $N_0=N_0(f)\geq 2$). The image $L_s=g^{N_0}(L_u)$ of $L_u$ under $G:=g^{N_0}|_{L_u}$ defines another lenticular region $L_s\subset R_{a_s}$ (whose boundary is also the union of pieces of the stable manifold of $p_s$ and the unstable manifold of $p_u$). The lenticular regions $g^{i}(L_u)$, $0\leq i\leq N_0$ are called \emph{parabolic tongues}.

\begin{figure}[htb!]
\includegraphics[scale=0.5]{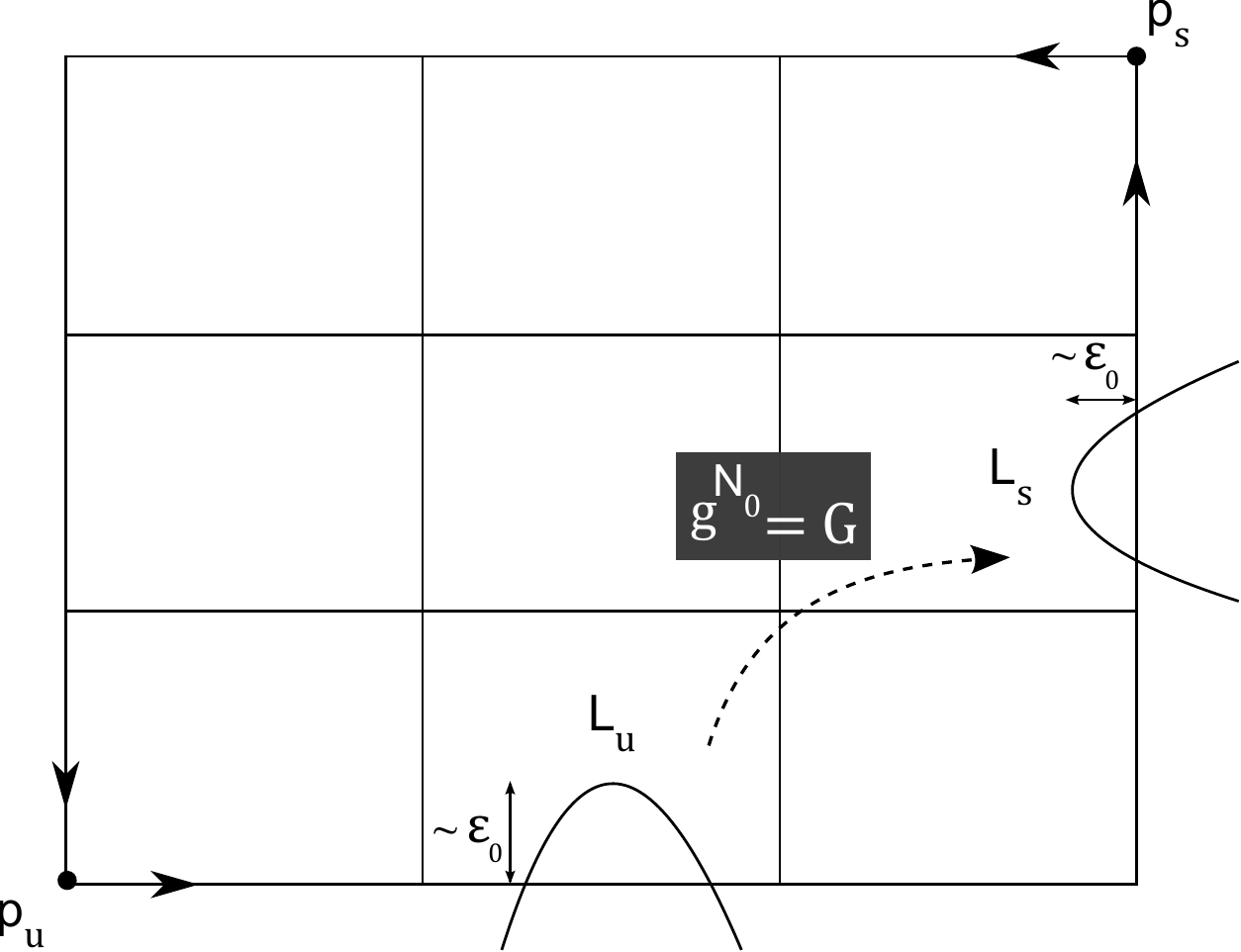}
\caption{Local dynamics near the parabolic tongues.}
\end{figure}

Let $\widehat{R}:=R\cup\bigcup\limits_{0<i<N_0} g^i(L_u)$. By definition, the set $\Lambda_g$ introduced in \eqref{e.Lambda-g} is the maximal invariant set of $\widehat{R}$, i.e., 
$$\Lambda_g=\bigcap\limits_{n\in\mathbb{Z}} g^{-n}(\widehat{R})$$
In other terms, the dynamics of $g\in\mathcal{U}_+$ on $\Lambda_g$ is localized in the region $\widehat{R}$ consisting of the Markov partition $R$ of the horseshoe $K_g$ and the parabolic tongues $g^i(L_u)$, $0<i<N_0$, associated to the unfolding of the heteroclinic tangency. 

Note that the dynamics of $g$ on $\widehat{R}$ is an \emph{iterated function system}, i.e., it is generated by the transition maps 
$$g_{aa'}=g|_{R_a\cap g^{-1}(R_{a'})}:R_a\cap g^{-1}(R_{a'})\to g(R_a)\cap R_{a'}, \quad (a,a')\in\mathcal{B},$$
related to the Markov partition $R$, and the folding map $G=g^{N_0}|_{L_u}:L_u\to L_s$ between the parabolic tongues.

In this language, we see that the transition maps $g_{aa'}$ behave like affine hyperbolic maps: for our choices of charts, $g_{aa'}$ contracts almost vertical directions and expands almost horizontal directions. Of course, this hyperbolic structure is not preserved by the folding map $G$ (as it might exchange almost horizontal and almost vertical directions) and this is the main source of non-hyperbolicity of $\Lambda_g$. 

In particular, it is not surprising that the definition in \cite{PY09} of non-uniformly hyperbolic horseshoes involves the features of a certain class of affine-like iterates of $g$. Before pursuing this direction, let us quickly remind the notion of \emph{affine-like maps}. 

\subsection{Affine-like maps} Let $I_0^s, I_0^u, I_1^s$ and $I_1^u$ be compact intervals and denote by $x_0,y_0, x_1$ and $y_1$ their corresponding coordinates. We say that a diffeomorphism $F$ from a \emph{vertical strip}
$$P:=\{(x_0,y_0):\varphi^-(y_0)\leq x_0\leq\varphi^+(y_0)\}\subset I_0^s\times I_0^u$$
onto a \emph{horizontal strip}
$$Q:=\{(x_1,y_1): \psi^-(x_1)\leq y_1\leq \psi^+(x_1)\}\subset I_1^s\times I_1^u$$
is \emph{affine-like} if the natural projection from the graph of $F$ to $I_0^u\times I_1^s$ is a diffeomorphism onto $I_0^u\times I_1^s$. 

By definition, an affine-like map $F$ has an \emph{implicit representation} $(A,B)$, i.e., there are smooth maps $A$ and $B$ on $I_0^u\times I_1^s$ such that $F(x_0,y_0)=(x_1,y_1)$ if and only if $x_0=A(y_0,x_1)$ and $y_1=B(y_0,x_1)$. 

In the context of $1$-parameter families $(g_t)_{|t|<t_0}$ generically unfolding heteroclinic bifurcations, we will consider exclusively affine-like maps satisfying a certain \emph{cone condition} and a certain \emph{distortion estimate}. 

More precisely, let $\lambda>1$, $u_0>0$, $v_0>0$ with 
$$1<u_0v_0\leq\lambda^2$$
and $D_0>0$ be the constants fixed in page 32 of \cite{PY09}: the choices of these constants depend only on the features of the initial diffeomorphism $f\in\mathcal{U}$.  

We say that an affine-like map $F(x_0,y_0)=(x_1,y_1)$ with implicit representation $(A,B)$ satisfies a \emph{cone condition} with parameters $(\lambda, u, v)$ whenever 
$$\lambda|A_x|+u_0|A_y|\leq 1 \quad \textrm{ and } \lambda|B_y|+v_0|B_x|\leq 1$$
where $A_x, A_y, B_x, B_y$ are the first order partial derivatives of $A$ and $B$. Also, we say that an affine-like map $F(x_0,y_0)=(x_1,y_1)$ with implicit representation $(A,B)$ satisfies a \emph{cone condition} with parameter $2D_0$ whenever the absolute values of the six functions
$$\partial_x\log|A_x|, \partial_y\log|A_x|, A_{yy}, \partial_y\log|B_y|, \partial_x\log|B_y|, B_{xx}$$
are uniformly bounded by $2D_0$. 

\begin{remark}\label{r.width-def.} Given an affine-like map $F:P\to Q$ with implicit representation $(A,B)$, we say that 
$$|P|:=\max|A_x| \quad \textrm{and} \quad |Q|:=\max|B_y|$$
are the widths of the domain $P$ and the image $Q$ of $F$. The widths have the property that $|P|\leq C\min|A_x|$ and $|Q|\leq C\min|B_y|$ where $C=C(f)\geq1$ is a constant depending only on $f\in\mathcal{U}$.
\end{remark}

The most basic examples of affine-like maps satisfying the cone and distortion conditions with parameters $(\lambda, u_0,v_0,2D_0)$ are the transition maps $g_{aa'}$ associated to the Markov partition $R$ of the horseshoe $K_g$ of $g\in\mathcal{U}$ (cf. Subsection 3.4 of \cite{PY09}). 

For our purposes, it is important to recall that new affine-like maps can be constructed from the so-called \emph{simple} and \emph{parabolic} compositions of two affine-like maps. 

Let $I_j^s, I_j^u$, $j=0,1,2$, be compact intervals and let $F:P\to Q$ and $F':P'\to Q'$ be two affine-like maps with domains $P\subset I_0^s\times I_0^u$ and 
$P'\subset I_1^s\times I_1^u$ and images $Q\subset I_1^s\times I_1^u$ and $Q'\subset I_2^s\times I_2^u$. Assume that both $F$ and $F'$ satisfy the cone condition with parameters $(\lambda, u_0, v_0)$. Then, the map $F''=F'\circ F$ from $P''=P\cap F^{-1}(P')$ to $Q''=Q\cap P'$ is an affine-like map satisfying the cone condition with parameters $(\lambda^2, u_0, v_0)$ called the \emph{simple composition} of $F$ and $F'$ (cf. Subsection 3.3 of \cite{PY09}). 

\begin{figure}[htb!]
\includegraphics[scale=0.5]{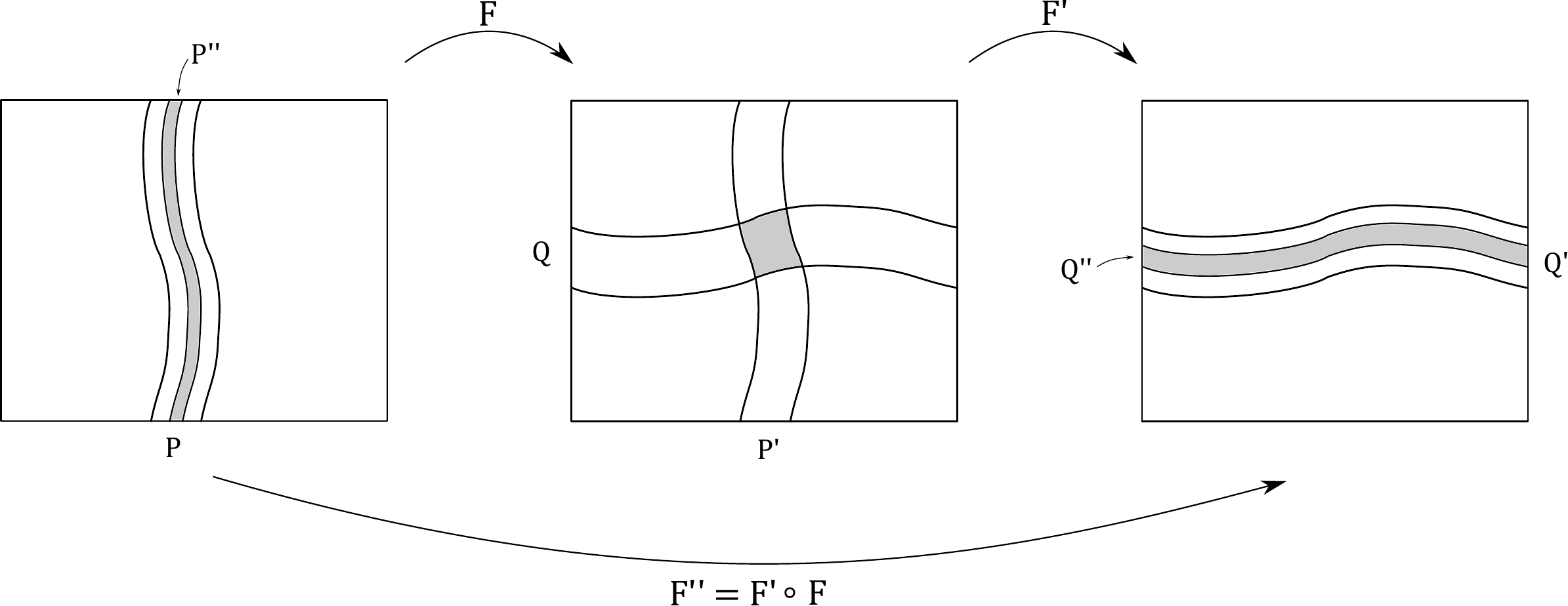}
\caption{Simple composition of affine-like maps.}
\end{figure}

Next, let $G$ be the folding map introduced in Subsection \ref{ss.local-dynamics} above (see also Subsection 2.3 of \cite{PY09}). Consider compact intervals $I_j^s, I_j^u$, $j=0,1$, and two affine-maps $F_0:P_0\to Q_0$, $F_1:P_1\to Q_1$ from vertical strips $P_0\subset I_0^s\times I_0^u$, $P_1\subset I_{a_s}^s\times I_{a_s}^u$ to horizontal strips $Q_0\subset I_{a_u}^s\times I_{a_u}^u$, $Q_1\subset I_1^s\times I_1^u$. As it is explained in Subsection 3.5 of \cite{PY09}, when a certain quantity $\delta(Q_0, P_1)$ roughly measuring the distance between $Q_0$ and the tip of the parabolic strip $G^{-1}(P_1)$ satisfies 
$$\delta(Q_0, P_1)>(1/b)(|P_1|+|Q_0|)$$
and the implicit representations of $F_0$ and $F_1$ to satisfy the bound $$\max\{|(A_1)_y|, |(A_1)_{yy}|, |(B_0)_x|, |(B_0)_{xx}|\}<b$$
for an adequate constant $b=b(f)>0$ depending only on $f\in\mathcal{U}$, the composition $F_1\circ G\circ F_0$ defines two affine-like maps $F^{\pm}:P^{\pm}\to Q^{\pm}$ with domains $P^{\pm}\subset P_0$ and $Q^{\pm}\subset Q_1$ called the \emph{parabolic compositions} of $F_0$ and $F_1$.

\begin{figure}[hbt!]
\includegraphics[scale=0.45]{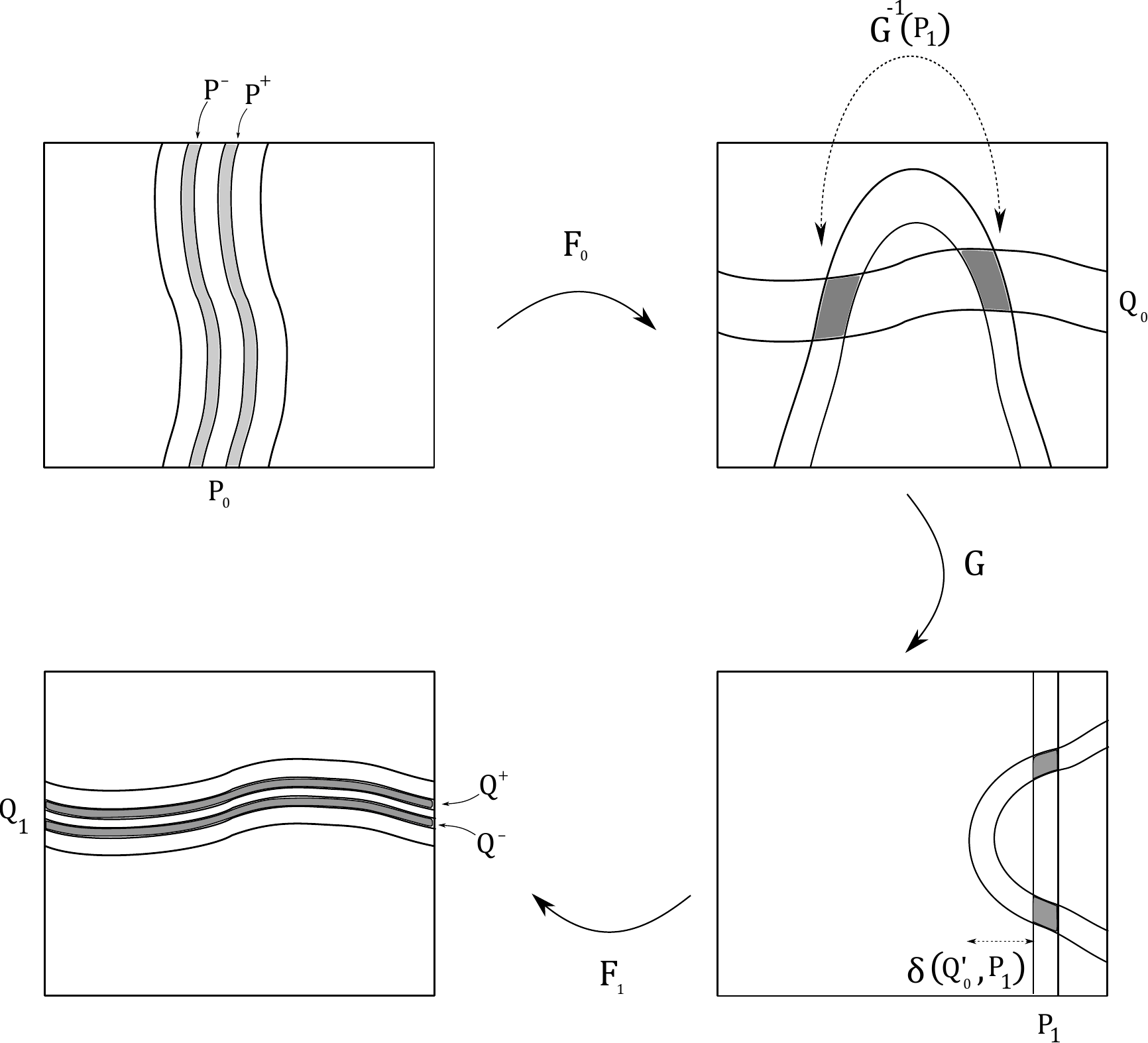}
\caption{Parabolic composition of affine-like maps.}
\end{figure}

\subsection{The class $\mathcal{R}(I)$ of certain $I$-persistent affine-like iterates} Coming back to the setting of Subsection \ref{ss.strong-regular-1}, let us consider again a $1$-parameter family $(g_t)_{|t|<t_0}$ generically unfolding a heteroclinic tangency (with normalized relative speed). Given $I\subset [\varepsilon_0,2\varepsilon_0]$ a parameter interval, we say that a triple $(P,Q,n)=(P_t, Q_t, n)_{t\in I}$ is a $I$-\emph{persistent affine-like iterate} whenever 
\begin{itemize}
\item $P_t\subset R_a$, $a\in\mathcal{A}$, is a vertical strip depending smoothly on $t\in I$, 
\item $Q_t\subset R_{a'}$, $a'\in\mathcal{A}$, is a horizontal strip depending smoothly on $t\in I$,
\item $n\geq 0$ is an integer such that, for all $t\in I$, the restriction $g_t^n|_{P_t}:P_t\to Q_t$ is an affine-like map (satisfying a cone condition) and, for each $0\leq m\leq n$, $g_t^m(P_t)\subset\widehat{R}$.
\end{itemize}

In Subsection 5.3 of \cite{PY09}, it is assigned to each candidate parameter interval $I$ a class $\mathcal{R}(I)$ of certain $I$-persistent affine-like iterates satisfying a list of seven conditions (R1) to (R7). 

Among these conditions, it is worth to point out that: 
\begin{itemize}
\item the transition maps 
$$g_{aa'}:R_a\cap g^{-1}(R_{a'})\to g(R_a)\cap R_{a'}, \quad (a,a')\in\mathcal{B},$$
belong to $\mathcal{R}(I)$ (cf. (R1) in \cite{PY09}), 
\item all $I$-persistent affine-like iterates $(P,Q,n)\in\mathcal{R}(I)$ satisfy the cone condition with parameters $(\lambda, u_0, v_0)$ and the distortion condition with parameter $2D_0$ (cf. (R2) in \cite{PY09}), 
\item the class $\mathcal{R}(I)$ is stable under simple compositions (cf. (R3) in \cite{PY09}) and certain \emph{allowed} parabolic compositions (cf. (R5) in \cite{PY09}), 
\item all $I$-persistent affine-like iterates $(P,Q,n)\in\mathcal{R}(I)$ with $n>1$ are obtained from simple or allowed parabolic compositions of shorter elements (cf. (R6) in \cite{PY09}),
\item if the parabolic composition of $(P_0, Q_0, n_0), (P_1, Q_1, n_1)\in\mathcal{R}(I)$ is allowed, then 
$$\delta(Q_0, P_1)\geq (1/C)(|P_1|^{1-\eta}+|Q_0|^{1-\eta})$$
where $\delta(Q_0, P_1)$ is the distance between $Q_0$ and the tip of $G^{-1}(P_1)$, $C=C(f)\geq 1$ is a constant depending only on $f\in\mathcal{U}$, and the parameter $\eta$ relates to $\varepsilon_0$ and $\tau$ via the condition $0<\varepsilon_0\ll\eta\ll\tau<1$ (cf. (R7) in \cite{PY09}).
\end{itemize} 
Furthermore, the class $\mathcal{R}(I)$ of $I$-persistent affine-like iterates satisfying the conditions (R1) to (R7) is \emph{unique} (cf. Theorem 1 of \cite{PY09}).

\subsection{Bicritical elements and strong regularity tests} In Subsection 5.6 of \cite{PY09}, the fundamental notion of $I$-\emph{bicritical element} $(P,Q,n)\in\mathcal{R}(I)$ is introduced: roughly speaking, a bicritical element corresponds to a return of the critical region of ``almost tangency'' to itself; in other terms, a bicritical element represents an ``almost tangency of higher order''. 

These bicritical elements present a potential danger to the non-uniform hyperbolicity features of $\Lambda_g$. Thus, it is not so surprising that the \emph{several quantitative requirements} in \cite{PY09} for a candidate parameter interval $I$ to pass the strong regularity test involve a precise control of the \emph{sizes} and \emph{numbers} of bicritical elements of $\mathcal{R}(I)$ in many \emph{scales} (cf. Definition 8 and the conditions $(\textrm{SR}1)$, $(\textrm{SR}2)'$ and $(\textrm{SR}3)_s$, $(\textrm{SR}3)_u$ in \cite{PY09}).

Among the \emph{qualitative} properties satisfied by a candidate interval $I$ passing the strong regularity test we have the $\beta$-\emph{regularity} property for some adequate choice of $\beta>1$. Concretely, the property of $\beta$-regularity for $I$ requires that all bicritical elements $(P, Q, n)\in\mathcal{R}(I)$ are thin in the sense that  
$$|P|<|I|^{\beta}, \quad |Q|<|I|^{\beta}.$$
See Definition 2 in \cite{PY09}. 

Concerning the choice of $\beta>1$, it depends only on the features of the initial diffeomorphism $f\in\mathcal{U}$: more precisely, one imposes the mild condition that \begin{equation}\label{e.beta-def-0}
1<\beta< 1 + \min\{\omega_s,\omega_u\}
\end{equation} 
where $\omega_s=-\frac{\log|\lambda(p_s)|}{\log|\mu(p_s)|}$ and $\omega_u=-\frac{\log|\mu(p_u)|}{\log|\lambda(p_u)|}$ with $\mu(p_s)$, $\mu(p_u)$ denoting the unstable eigenvalues of the periodic points $p_s, p_u$ and $\lambda(p_s)$, $\lambda(p_u)$ denoting the stable eigenvalues of the periodic points $p_s, p_u$, and the important condition that 
\begin{equation}\label{e.beta-def}1<\beta<\frac{(1-\min\{d_s^0, d_u^0\})(d_s^0+d_u^0)}{\max\{d_s^0, d_u^0\}(\max\{d_s^0, d_u^0\} + d_s^0+d_u^0-1)}:=\beta^*(d_s^0,d_u^0)
\end{equation}
(cf. Remark 8 in \cite{PY09}). 

A detailed study of the strong regularity tests is performed in Section 9 of \cite{PY09}, where it is shown that most candidate intervals pass the strong regularity tests: namely, the relative measure of the union of the candidate intervals $I\subset I_0:=[\varepsilon_0, 2\varepsilon_0]$ failing the strong regularity tests is $\leq 3\varepsilon_0^{\tau^2}$ (cf. Corollary 15 of \cite{PY09}).

\subsection{Non-uniformly hyperbolic horseshoes}

Once we know that most candidate intervals are strongly regular, let us quickly review the relationship between strong regularity and non-uniform hyperbolicity. 

For this sake, we fix \emph{once and for all} a strongly regular parameter $t\in I_0=[\varepsilon_0, 2\varepsilon_0]$, say $\{t\}=\bigcap\limits_{m\in\mathbb{N}} I_m$ for some decreasing sequence $I_m$ of candidate intervals passing the strong regularity tests, and we denote by $g_t=g\in\mathcal{U}_+$ the corresponding dynamical system. 

Consider the class $\mathcal{R}:=\bigcup\limits_{m\in\mathbb{N}}\mathcal{R}(I_m)$ of certain affine-like iterates of $g$. Given a decreasing sequence of vertical strips $P_k$ associated to some affine-like iterates $(P_k, Q_k, n_k)\in\mathcal{R}$, we say that $\omega=\bigcap\limits_{k\in\mathbb{N}} P_k$ is a \emph{stable curve}. 

The set of stable curves is denoted by $\mathcal{R}^{\infty}_+$. The union of stable curves 
$$\widetilde{\mathcal{R}}^{\infty}_+:=\bigcup\limits_{\omega\in\mathcal{R}^{\infty}_+}\omega$$ 
is a lamination by $C^{1+Lip}$ (stable) curves with Lipschitz holonomy (cf. Subsection 10.5 of \cite{PY09}). 

The set $\mathcal{R}^{\infty}_+$ of stable curves has a natural partition defined in terms of the notion of \emph{prime elements} of $\mathcal{R}$. More precisely, we say that $(P,Q,n)\in\mathcal{R}$ is a prime element if it is not the simple composition of two shorter elements. Using this concept, we can write  $\mathcal{R}^{\infty}_+ := \mathcal{D}^{\infty}_+\cup\mathcal{N}_+$ where $\mathcal{N}^{\infty}_+$ is the set of stable curves contained in infinitely many prime elements and $\mathcal{D}^{\infty}_+$ is the complement of $\mathcal{N}_+$. 

The partition $\mathcal{R}^{\infty}_+ = \mathcal{D}_+\cup\mathcal{N}_+$ allows to partially define an induced dynamics $T^+$ on $\mathcal{R}^{\infty}_+$: given a stable curve $\omega\in\mathcal{D}_+$ and denoting by $(P,Q,n)\in\mathcal{R}$ the thinnest prime element containing $\omega$, one can show that $g^n(\omega)$ $T^+(\omega)$ is contained in a stable curve $\omega':=T^+(\omega)\in\mathcal{R}^{\infty}_+$. 

The map $T^+:\mathcal{D}_+\to\mathcal{R}^{\infty}_+$ is Bernoulli and uniformly expanding with countably many branches (cf. Subsection 10.5 of \cite{PY09}). Furthermore, one has a natural $1$-parameter family of transfer operators $L_d$ associated to $T^+$ whose dominant eigenvalues $\lambda_d>0$ detect the transverse Hausdorff dimension of the lamination $\widetilde{\mathcal{R}}^{\infty}_+$: more precisely, $\widetilde{\mathcal{R}}^{\infty}_+$ has Hausdorff dimension $1+d_s$ where $d_s$ is the unique value of $d$ with $\lambda_d=1$ (cf. Theorem 4 of \cite{PY09}). Also, the map $T^+$ captures most of the dynamics on $\mathcal{R}^{\infty}_+$ because most stable curves can be indefinitely iterated under $T^+$: denoting by $\mathcal{D}^{\infty}_+:=\bigcap\limits_{j\geq 0}(T^+)^{-j}(\mathcal{D}_+)$ and $\widetilde{\mathcal{D}}^{\infty}_+:=\bigcup\limits_{\omega\in \mathcal{D}^{\infty}_+}\omega$, the transverse Hausdorff dimension of $\widetilde{\mathcal{R}}^{\infty}_+ - \widetilde{\mathcal{D}}^{\infty}_+$ is $<d_s$ (cf. Proposition 57 of \cite{PY09}).

The properties described in the previous paragraph justify calling 
$$\{z\in W^s(\Lambda): g^n(z)\in\widetilde{\mathcal{R}}^{\infty}_+ \textrm{ for some } n\geq 0\}$$
the \emph{well-behaved part} of the stable set $W^s(\Lambda_g)$. 

In a similar vein, the unstable set $W^u(\Lambda_g)$ also has a well-behaved part consisting of all points whose orbit eventually enters the lamination $\widetilde{\mathcal{R}}^{\infty}_-$ consisting of unstable curves (decreasing intersections of horizontal strips associated to affine-like iterates in $\mathcal{R}$).

The nomenclature \emph{non-uniformly hyperbolic horseshoe} for $\Lambda=\Lambda_g$ is justified in \cite{PY09} by showing that the exceptional set of points outside the well-behaved part of $W^s(\Lambda)$ has the following properties: 
\begin{itemize}
\item it intersects each unstable curve in a subset of Hausdorff dimension $<d_s$ (cf. Subsection 11.5 of \cite{PY09}), and   
\item its $2$-dimensional Lebesgue measure is zero (cf. Subsection 11.6 and also Theorem 7 of \cite{PY09}).
\end{itemize}

\subsection{The stable set of a non-uniformly hyperbolic horseshoe} Following the Subsection 11.6 of \cite{PY09}, we write the stable set $W^s(\Lambda)$ as the countable union of dynamical copies of the local stable set $W^s(\Lambda, \widehat{R})\cap R$, i.e., 
$$W^s(\Lambda)=\bigcup\limits_{n\geq 0} g^{-n}(W^s(\Lambda,\widehat{R})\cap R))$$
and we split the local stable set $W^s(\Lambda,\widehat{R})\cap R$ into its well-behaved part and its \emph{exceptional part}:
$$W^s(\Lambda,\widehat{R})\cap R := \bigcup\limits_{n\geq 0} \left(W^s(\Lambda,\widehat{R})\cap R\cap g^{-n}(\widetilde{\mathcal{R}}^{\infty}_+)\right)\cup\mathcal{E}^+$$
where 
\begin{equation}\label{e.exceptional-set-def}
\mathcal{E}^+:=\{z\in W^s(\Lambda,\widehat{R})\cap R: g^n(z)\notin\widetilde{\mathcal{R}}^{\infty}_+ \textrm{ for all } n\geq 0\}
\end{equation}
Since $g$ is a diffeomorphism and the $C^{1+Lip}$-lamination $\widetilde{\mathcal{R}}^{\infty}_+$ has transverse Hausdorff dimension $0<d_s<1$ (cf. Theorem 4 of \cite{PY09}), we deduce that the Hausdorff dimension of the stable set $W^s(\Lambda)$ is:
\begin{proposition}\label{p.HD-Ws-E+} $HD(W^s(\Lambda))=\max\{1+d_s, HD(\mathcal{E}^+)\}$.
\end{proposition}

For our purposes of studying the quantity $HD(\mathcal{E}^+)$, it is useful to recall that the exceptional set $\mathcal{E}^+$ has a natural decomposition in terms of the successive passages through the so-called  \emph{parabolic cores} of vertical strips (cf. Subsection 11.7 of \cite{PY09}). 

More precisely, given $(P,Q,n)\in\mathcal{R}$, the parabolic core $c(P)$ is the set of points of $W^s(\Lambda,\widehat{R})$ belonging to $P$ but not to any \emph{child} of $P$. Here, a child\footnote{The child terminology in page 33 of \cite{PY09} is not exactly the one given above, but it is shown in Section 6.2 of \cite{PY09} that these two definitions coincide.} $P'$ of $P$ means a vertical strip associated to some element $(P',Q',n')\in\mathcal{R}$ obtained by simple compositions with the transition maps $g_{aa'}$ of the Markov partition of the horseshoe $K_g$ or parabolic composition of $(P,Q,n)\in\mathcal{R}$ with some element of $\mathcal{R}$. 

The set of elements $(P_0, Q_0, n_0)\in\mathcal{R}$ with $c(P_0)\neq\emptyset$ is denoted by $\mathcal{C}_-$. By definition, one can write 
$$\mathcal{E}^+=\bigcup\limits_{(P_0, Q_0, n_0)\in\mathcal{C}_-}\mathcal{E}^+(P_0)$$
where 
$$\mathcal{E}^+(P_0):=\mathcal{E}^+\cap c(P_0)$$

Recall that $c(P_0)\neq\emptyset$ for some $(P_0,Q_0,n_0)\in\mathcal{R}$ implies that 
$$g^{n_0}(\mathcal{E}^+(P_0))\subset Q_0\cap L_u\cap\mathcal{E}^+$$ and 
$$G(g^{n_0}(\mathcal{E}^+(P_0))=g^{n_0+N_0}(\mathcal{E}^+(P_0))\subset L_s\cap\mathcal{E}^+.$$
This permits to decompose each $\mathcal{E}^+(P_0)$ as 
$$\mathcal{E}^+(P_0):=\bigcup\limits_{(P_1,Q_1,n_1)\in\mathcal{C}_-}\mathcal{E}^+(P_0, P_1)$$
where 
$$\mathcal{E}^+(P_0, P_1):=\{z\in\mathcal{E}^+(P_0): g^{n_0+N_0}(z)\in c(P_1)\}$$ 

In general, for each $k\in\mathbb{N}$, we can inductively define a decomposition
$$\mathcal{E}^+(P_0,\dots, P_k)=\bigcup\limits_{(P_{k+1}, Q_{k+1}, n_{k+1})\in\mathcal{C}_-} \mathcal{E}^+(P_0,\dots, P_k, P_{k+1})$$
and, consequently, 
$$\mathcal{E}^+=\bigcup\limits_{(P_0, P_1, \dots, P_k) \textrm{ admissible }}\mathcal{E}^+(P_0,\dots, P_k)$$
where $(P_0,\dots,P_k)$ is \emph{admissible} whenever $\mathcal{E}^+(P_0,\dots, P_k)\neq\emptyset$. The admissibility condition on $(P_0,\dots,P_{k+1})$ imposes severe restrictions on the elements $(P_i, Q_i, n_i)\in\mathcal{R}$: for example, $(P_0,Q_0,n_0)\in\mathcal{C}_-$, 
\begin{equation}\label{e.11.64}
\max\{|P_1|, |Q_1|\}\leq \varepsilon_0^{\beta}
\end{equation}
and, by setting $\widetilde{\beta}:=\beta(1-\eta)(1+\tau)^{-1}$,  
\begin{equation}\label{e.Lemma24}
\max\{|P_{j+1}|, |Q_{j+1}|\}\leq C|Q_j|^{\widetilde{\beta}}
\end{equation}
for all $j\geq 1$ (cf. Lemma 24 of \cite{PY09}). 
\begin{remark} Here and in what follows $C=C(f)\geq 1$ denotes an appropriate large constant depending only on $f\in\mathcal{U}$.
\end{remark}
In particular, by taking $1<\widehat{\beta}<\widetilde{\beta}$, the admissibility condition forces that 
\begin{equation}\label{e.Lemma24'}
\max\{|P_j|,|Q_j|\}\leq\varepsilon_0^{\widehat{\beta}^j}
\end{equation} 
(for $\varepsilon_0$ sufficiently small), that is, the widths of the strips $P_j$ and $Q_j$ confining the dynamics of $\mathcal{E}^+$ decay doubly exponentially fast. 

In the sequel, we will use the decomposition 
$$\mathcal{E}^+=\bigcup\limits_{(P_0, P_1, \dots, P_k) \textrm{ admissible }}\mathcal{E}^+(P_0,\dots, P_k)$$ 
in order to estimate/compute the Hausdorff dimension of $\mathcal{E}^+$. 

\begin{remark}\label{r.time-symmetric} The arguments in \cite{PY09} are \emph{time-symmetric} in the sense that all definitions and results above about $W^s(\Lambda)$ have natural counterparts for $W^u(\Lambda)$ (after exchanging the roles of past and future, vertical strips $P$ and horizontal strips $Q$, etc.). In particular, in the proof of Theorem \ref{t.MPY-A}, it suffices to study $W^s(\Lambda)$.
\end{remark}

\subsection{Some notations for Hausdorff measures} For later use, we use the following notations. Let $X$ a bounded subset of the plane. Given $0\leq d\leq 2$, and $\delta>0$, we write $m^d_{\delta}(X)$ for the infimum over open coverings $(U_i)_{i\in I}$ of $X$ with diameter $\textrm{diam}(U_i)<\delta$ of the following quantity
$$\sum\limits_{i\in I}\textrm{diam}(U_i)^d.$$
In other terms, $m_{\delta}^d(X)$ is the $d$-Hausdorff measure at scale $\delta>0$ of $X$. Note that if $X$ is a finite or countable union $\bigcup\limits_{\alpha}X_{\alpha}$, we obviously have 
$$m_{\delta}^d(X)\leq\sum\limits_{\alpha} m_{\delta}^d(X_{\alpha}).$$

In this language, the $d$-Hausdorff measure is $m^d(X)=\lim\limits_{\delta\to 0} m_{\delta}^d(X)$ and the Hausdorff dimension is 
$$HD(X):=\inf\{d\in[0, 2]: m^d(X)=0\}$$
\section{The stable set $W^s(\Lambda)$ has Hausdorff dimension $<2$}\label{s.MPY-A}

By Proposition \ref{p.HD-Ws-E+} (and Remark \ref{r.time-symmetric}), the proof of Theorem \ref{t.MPY-A} is reduced to the following statement:

\begin{theorem}\label{t.MPY-A-1} $\textrm{HD}(\mathcal{E}^+)<2$.
\end{theorem}

We begin the proof of Theorem \ref{t.MPY-A-1} by showing the following lemma:
\begin{lemma}\label{l.mdE+P0...Pk} Fix $1<\widehat{\beta}<\widetilde{\beta}:=\beta(1-\eta)(1+\tau)^{-1}$ and let $s_k:=C^{k+1}\varepsilon_0^{\frac{(1-\eta)}{2}\widehat{\beta}^k}$. Then, the $d$-Hausdorff measure at scale $s_k$ of 
$\mathcal{E}^+(P_0,\dots,P_k)$ satisfies 
$$m^d_{s_k}(\mathcal{E}^+(P_0,\dots, P_k))\leq C^{4k+2} |P_0|^{d-1}\cdot\prod\limits_{j=0}^{k-2}\frac{|P_{j+1}|^{d-1}}{|Q_j|}\cdot\frac{|P_k|^{d-1}}{|Q_{k-1}|^{\frac{1+\eta}{2}}}\cdot |Q_k|^{\frac{(d-1)(1-\eta)}{2}}$$
for any $1\leq d\leq 2$ and any admissible $(P_0,\dots,P_k)$. 
\end{lemma}

\begin{proof} Let $(P_0,\dots, P_k)$ be admissible. By definition, 
$$g^{n_0+N_0+\dots+n_{k-1}+N_0}(\mathcal{E}^+(P_0,\dots,P_k))\subset g^{n_{k-1}+N_0}(c(P_{k-1}))\cap c(P_k).$$ 

Since $c(P_j)\neq\emptyset$, i.e., $(P_j,Q_j,n_j)\in\mathcal{C}_-$, for all $j=0, \dots, k$, it follows from Proposition 62 of \cite{PY09} that:
\begin{itemize}
\item[(a)] $g^{n_{k-1}+N_0}(c(P_{k-1}))$ is a subregion of $g^{N_0}(Q_{k-1}\cap L_u)$ with diameter $\leq C|Q_{k-1}|^{\frac{(1-\eta)}{2}}$, and, a fortiori, $g^{n_{k-1}+N_0}(c(P_{k-1}))$ is contained in a horizontal strip of width $C|Q_{k-1}|^{\frac{(1-\eta)}{2}}$
\item[(b)] $g^{n_k}(c(P_k))$ is a subregion of $Q_k\cap L_u$ of diameter $\leq C|Q_k|^{\frac{(1-\eta)}{2}}$. 
\end{itemize}
On the other hand, the affine-like iterate $g^{n_k}|_{P_k}$ expands the horizontal direction by a factor $\sim 1/|P_k|$ and contracts the vertical direction by a factor $\sim |Q_k|$. It follows from item (b) above that $c(P_k)$ is contained in a vertical strip of width $C|Q_k|^{\frac{(1-\eta)}{2}}|P_k|$. 

\begin{figure}[htb!]
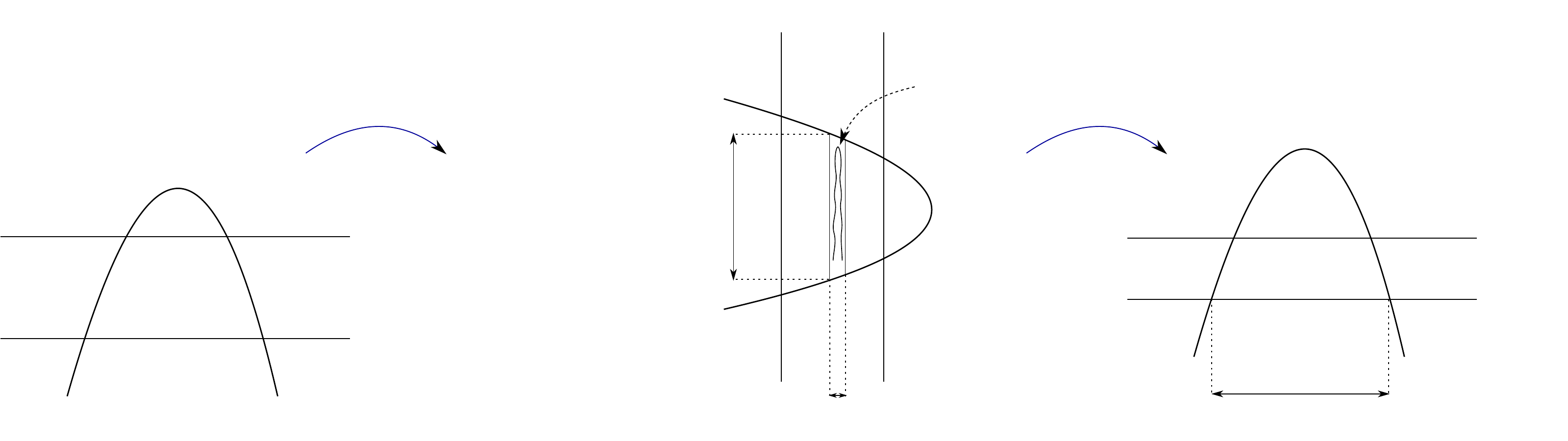
\end{figure}

In particular, we deduce that $g^{n_0+N_0+\dots+n_{k-1}+N_0}(\mathcal{E}^+(P_0,\dots,P_k))$ is contained in a rectangular region of dimensions $C|Q_k|^{\frac{(1-\eta)}{2}}|P_k|$ (in the horizontal direction) and $C|Q_{k-1}|^{\frac{(1-\eta)}{2}}$ (in the vertical direction). By partitioning this rectangular region into $\widetilde{N}_k:= C^2 \frac{|Q_{k-1}|^{(1-\eta)/2}}{|Q_k|^{(1-\eta)/2} |P_k|}$ squares of sides of lengths $C|Q_k|^{\frac{(1-\eta)}{2}}|P_k|$, we obtain a covering $\mathcal{O}_k$ of 
$$g^{n_0+N_0+\dots+n_{k-1}+N_0}(\mathcal{E}^+(P_0,\dots,P_k))$$ by $\widetilde{N}_k$ squares whose sides have length $C|Q_k|^{\frac{(1-\eta)}{2}}|P_k|$. 

From the covering $\mathcal{O}_k$, we produce a covering of $\mathcal{E}^+(P_0,\dots,P_k)$ by analyzing \emph{individually} the evolution of the elements of $\mathcal{O}_k$ under the backward iterates of $g$ using the following \emph{inductive} procedure. 

At the $i$-th step, we have $M_{i}:=\prod\limits_{j=k-i}^k \widetilde{N}_j$ squares of dimensions $\delta_i:=C^{i+1}|Q_k|^{\frac{(1-\eta)}{2}}\prod\limits_{j=k-i}^k |P_j|$ forming a covering 
$\mathcal{O}_{k-i}$ of 
$$g^{n_0+N_0+\dots+n_{k-i-1}+N_0}(\mathcal{E}^+(P_0,\dots,P_k))$$ 
Since the affine-like iterate $g^{n_{k-i-1}}|_{P_{k-i-1}}$ expands the horizontal direction by a factor $\sim 1/|P_{k-i-1}|$ and contracts the vertical direction by a factor $\sim |Q_{k-i-1}|$ (and the folding map $G=g^{N_0}$ is fixed), we see that the inverse of $g^{n_{k-i-1}+N_0}|_{P_{k-i-1}}$ sends each element of $\mathcal{O}_{k-i}$ into a rectangular region of dimensions $\delta_{i+1}:=C\cdot \delta_i\cdot |P_{k-i-1}|$ (in the horizontal direction) and $C\cdot \delta_i/|Q_{k-i-1}|$ (in the vertical direction). By partitioning each of these rectangular regions into $\widetilde{N}_{k-i-1}:=C^2/(|P_{k-i-1}|\cdot |Q_{k-i-1}|)$ squares with sides of length $\delta_{i+1}$, we get a covering $\mathcal{O}_{k-i-1}$ of 
$$g^{n_0+N_0+\dots+n_{k-i-2}+N_0}(\mathcal{E}^+(P_0,\dots,P_k))$$
by $M_{i+1}:=\widetilde{N}_{k-i-1}\cdot M_{i}$ squares whose sides have length $\delta_{i+1}$.

In the end of the $(k-1)$-th step of this procedure, we obtain a covering $\mathcal{O}_0$ of $\mathcal{E}^+(P_0,\dots,P_k)$ by $M_k$ squares of sides of length $\delta_k$. 

Observe that, by \eqref{e.Lemma24'}, one has $\delta_k\leq C^{k+1}|Q_k|^{\frac{(1-\eta)}{2}}\leq C^{k+1}\varepsilon_0^{\frac{(1-\eta)}{2}\widehat{\beta}^k}:=s_k$. In particular, we can use the covering $\mathcal{O}_0$ of $\mathcal{E}^+(P_0,\dots,P_k)$ to get the estimate
\begin{eqnarray*}
& &m_{s_k}^d(\mathcal{E}^+(P_0,\dots,P_k)) \leq M_k\delta_k^{d}=C^{2(k-1)}\prod\limits_{j=0}^{k-1}\frac{1}{|P_j|\cdot |Q_j|}\widetilde{N}_k\delta_k^d \\ 
& &\leq C^{2k+(k+1)d} \prod\limits_{j=0}^{k-1}\frac{1}{|P_j|\cdot |Q_j|}\cdot \frac{|Q_{k-1}|^{(1-\eta)/2}}{|Q_k|^{(1-\eta)/2} |P_k|}\cdot |Q_k|^{d(1-\eta)/2}\prod\limits_{j=0}^k |P_j|^d \\
& &\leq C^{4k+2} \prod\limits_{j=0}^{k-1}\frac{|P_j|^{d-1}}{|Q_j|} |P_k|^{d-1} |Q_{k-1}|^{(1-\eta)/2}|Q_k|^{(d-1)(1-\eta)/2} \\ 
& & = C^{4k+2} |P_0|^{d-1}\cdot\prod\limits_{j=0}^{k-2}\frac{|P_{j+1}|^{d-1}}{|Q_j|}\cdot\frac{|P_k|^{d-1}}{|Q_{k-1}|^{(1+\eta)/2}}\cdot |Q_k|^{(d-1)(1-\eta)/2}
\end{eqnarray*}

This proves the lemma.
\end{proof} 

\begin{remark}\label{r.MPY-A-idea} The basic idea to prove the previous lemma is the following. We start by covering an adequate forward iterate of $\mathcal{E}^+(P_0,\dots,P_k)$ with little squares and we bring back this covering under the dynamics. The affine-like iterates will stretch these squares into vertical strips and, \emph{before} the folding map $G=g^{N_0}$ ``bend'' these strips (making their geometry very intricate), we subdivide each strip into smaller squares in order to keep a \emph{qualitative} control of the covering after the folding map acts. 

Of course, we lose control of the \emph{fine} geometrical structure of $\mathcal{E}^+(P_0,\dots,P_k)$ in this argument and this is why one can not hope to apply Lemma \ref{l.mdE+P0...Pk} to get the expected Hausdorff dimension for $\mathcal{E}^+$ (see also Remark \ref{r.expected-dimension-MPY-A} below).  
\end{remark}

The next lemma says that the estimate in Lemma \ref{l.mdE+P0...Pk} is particularly useful when the parameter $1\leq d\leq 2$ is chosen close to $2$:

\begin{lemma}\label{l.mdE+P0...Pk2} Assume that $7/5<d<2$ and $\widetilde{\beta}(d-1)>1$ (where $\widetilde{\beta}=\beta(1-\eta)(1+\tau)^{-1}$ and $\beta$ satisfies \eqref{e.beta-def-0} and \eqref{e.beta-def}). Then, in the setting of Lemma \ref{l.mdE+P0...Pk}, one has 
$$m^d_{s_k}(\mathcal{E}^+(P_0,\dots,P_k))\leq C^{5k+2}|P_0|^{d-1}|Q_k|^{d^-}$$ 
where $d^->d_s+d_u-1$ is close to $d_s+d_u-1$ (and $d_s$, $d_u$ are the transverse Hausdorff dimensions of the well-behaved parts of $W^s(\Lambda)$ 
and $W^u(\Lambda)$).
\end{lemma}

\begin{proof} Since $\widetilde{\beta}(d-1)>1>(1+\eta)/2$ (as $0<\eta\ll 1$), we can combine\footnote{Together with \eqref{e.11.64} and the argument at page 193 of \cite{PY09} (right after (11.73)) to deal with the special case $|P_1|^{d-1}/|Q_0|$.} \eqref{e.Lemma24} with Lemma \ref{l.mdE+P0...Pk} to deduce that 
$$m^d_{s_k}(\mathcal{E}^+(P_0,\dots,P_k))\leq C^{5k+1}|P_0|^{d-1}|Q_k|^{\frac{(d-1)(1-\eta)}{2}}$$

Fix $d_s+d_u-1<d^-<1/5$ close to $d_s+d_u-1$: this choice is possible because the dimension condition \eqref{e.PYset} ensures that $d_s+d_u<6/5$. In particular,  $(d-1)(1-\eta)/2>1/5>d^-$ (as $d>7/5$ and $0<\eta\ll1$) and hence 
$$m^d_{s_k}(\mathcal{E}^+(P_0,\dots,P_k))\leq C^{5k+1}|P_0|^{d-1}|Q_k|^{d^-}$$
This shows the lemma.
\end{proof}

At this point, we are ready to complete the proof of Theorem \ref{t.MPY-A-1} (and, a fortiori, Theorem \ref{t.MPY-A}).

\begin{proof}[Proof of Theorem \ref{t.MPY-A-1}] Fix $7/5<d<2$ close to $2$ so that $\widetilde{\beta}(d-1)>1$ and $d>1+d_s^0(1+C\varepsilon_0^{\tau})$. 

For each $k\in\mathbb{N}$, consider the decomposition
$$\mathcal{E}^+=\bigcup\limits_{(P_0,\dots,P_k) \textrm{ admissible }}\mathcal{E}^+(P_0,\dots,P_k)$$

Since $7/5<d<2$ and $\widetilde{\beta}(d-1)>1$, we can apply Lemma \ref{l.mdE+P0...Pk2} to deduce that 
$$m^d_{s_k}(\mathcal{E}^+)\leq C^{5k+2}\sum\limits_{(P_0,\dots,P_k) \textrm{ admissible }}|P_0|^{d-1}|Q_k|^{d^-}$$
where $d^->d_s^0+d_u^0-1$ is close to $d_s^0+d_u^0-1$.

As it was shown in page 193 of \cite{PY09} (after (11.77)), the quantity of admissible sequences $(P_0,\dots, Q_0)$ with fixed extremities $P_0$ and $Q_k$ is $\leq C|Q_k|^{-C\eta}$. Since the admissibility condition implies that $Q_0, \dots, Q_k$ are critical (in the sense of the definition at page 37 of \cite{PY09}), we get from the previous estimate that 
$$m^d_{s_k}(\mathcal{E}^+)\leq C^{5k+3}\sum\limits_{\substack{P_0 \textrm{ with } Q_0 \textrm{critical}, \\ Q_k \textrm{ critical}}}|P_0|^{d-1}|Q_k|^{d^- - C\eta}$$
On the other hand, we know from Subsection 11.5.10 of \cite{PY09} that 
$$\sum\limits_{P \textrm{ with } Q \textrm{ critical }}|P|^{\rho_s}\leq C<\infty$$
where $d_s^0+C\varepsilon_0<\rho_s<d_s^0(1+C\varepsilon_0^{\tau})$ is a parameter defined at pages 135 and 138 of \cite{PY09}, and from Subsection 11.5.9 of \cite{PY09} that 
$$\sum_{Q \textrm{ critical }}|Q|^{d^- - 2C\eta}\leq C<\infty$$

Since $d-1>d_s^0(1+C\varepsilon_0^{\tau})>\rho_s$ and $|Q_k|\leq\varepsilon_0^{\widehat{\beta}^k}$ (cf. \eqref{e.Lemma24'}), we conclude that 
$$m^d_{s_k}(\mathcal{E}^+)\leq C^{5k+5}\varepsilon_0^{C\eta\widehat{\beta}^k}$$
Because $\widehat{\beta}>1$, by letting $k\to\infty$, we see that the Hausdorff measures 
$$m^d_{s_k}(\mathcal{E}^+)\leq C^{5k+5}\varepsilon_0^{C\eta\widehat{\beta}^k}\to 0$$ 
along a sequence of scales $s_k:=C^{k+1}\varepsilon_0^{\frac{(1-\eta)}{2}\widehat{\beta}^k}\to 0$. 

Therefore, $HD(\mathcal{E}^+)\leq d<2$. This proves Theorem \ref{t.MPY-A-1} (and also Theorem \ref{t.MPY-A} in view of Proposition \ref{p.HD-Ws-E+}).
\end{proof}

\begin{remark}\label{r.expected-dimension-MPY-A} Still along the lines of Remark \ref{r.MPY-A-idea} above, let us observe that the argument used in the proof of Theorem \ref{t.MPY-A-1} does \emph{not} allow us to show that $\textrm{HD}(\mathcal{E}^+)<1+d_s$ when $d_s^0+d_u^0>1$. Indeed, among the several conditions imposed on the parameter $d$ with $HD(\mathcal{E}^+)\leq d<2$, we required that $\widetilde{\beta}(d-1)=\beta(1-\eta)(1+\tau)^{-1}(d-1)>1$ and $d-1>d_s^0(1+o(1))$. Because $d_s-d_s^0=o(1)$ and $0<\eta\ll\tau\ll1$, if we want to take $d<1+d_s$, then the inequality $\beta d_s^0>1$ must hold. However, it is \emph{never} the case that $\beta d_s^0>1$ when $d_s^0+d_u^0>1$: indeed, by \eqref{e.beta-def}, one has $\beta \max\{d_s^0,d_u^0\}<\frac{(1-\min\{d_s^0,d_u^0\})(d_s^0+d_u^0)}{\max\{d_s^0,d_u^0\}+d_s^0+d_u^0-1} = \frac{(d_s^0+d_u^0)-\min\{d_s^0,d_0^u\}(d_s^0+d_u^0)}{(d_s^0+d_u^0)-(1-\max\{d_s^0,d_u^0\})}<1$. 
\end{remark}

\end{document}